\documentclass[11pt,letterpaper,margin=1cm]{amsart}
\usepackage[utf8]{inputenc}
\usepackage{amssymb,amsmath, amsthm, mathabx}
\usepackage{color, enumerate}
\usepackage{hyperref}
 \hypersetup{
	colorlinks=true,
	linkcolor=blue,
	filecolor=blue,
	citecolor = red,      
	urlcolor=cyan,
}

\usepackage{bm}
\usepackage{caption}
\usepackage{subcaption}
\usepackage{a4wide}
\usepackage{accents}
\usepackage{tikz}
\usepackage{graphicx}
\usepackage{chngcntr}
\usepackage{nicefrac}
\usepackage[style=numeric,backend=bibtex,date=year,sortcites]{biblatex}
\renewbibmacro{in:}{} %for authoryear style, to suppress the In:

\usepackage{algorithm}
\usepackage{algpseudocode}

\counterwithin{figure}{section}
\counterwithin{equation}{section}

\usepackage{todonotes, xargs}
\newcommandx{\note}[2][1=]{\todo[linecolor=yellow,backgroundcolor=yellow!25,bordercolor=yellow,#1]{#2}}

\newcommand{\rb}[1]{\left( #1 \right)}
\newcommand{\set}[1]{ \left\{ #1 \right\} }

\renewcommand{\a}{\ensuremath{\mathcal{A}}}
\renewcommand{\b}{\ensuremath{\mathcal{B}}}
\newcommand{\mcc}{\ensuremath{\mathcal{C}}}
\newcommand{\ee}{\mathbb{E}}
\newcommand{\g}{\ensuremath{\mathcal{G}}}
\newcommand{\n}{\mathcal{N}}
\newcommand{\nn}{\mathbb{N}}

\newcommand{\pp}{\mathbb{P}}

\newcommand{\rr}{\mathbb{R}}
\newcommand{\s}{\mathcal{S}}
\renewcommand{\t}{\mathcal{T}}
\newcommand{\x}{\mathcal{X}}

\newcommand{\zz}{\mathbb{Z}}

\newcommand{\nonnegint}{\zz_{\geq 0}}
\newcommand{\polring}[2]{\rr[#1_1,\ldots,#1_{#2}]}

\newcommand{\wh}[1]{\widehat{#1}}

\newcommand{\init}[2]{\text{in}_{#1}\left(#2\right)}
\newcommand{\ip}[2]{\ensuremath{(IP)_{#1, #2} }}

\newcommand{\argmin}{\operatorname{argmin}}

\allowdisplaybreaks

\newtheorem{thm}{Theorem}[section]
\newtheorem{prop}[thm]{Proposition}
\newtheorem{lem}[thm]{Lemma}

\theoremstyle{remark}
\newtheorem{defn}[thm]{Definition}

\newtheorem{example}[thm]{Example}

\title{Gröbner and Graver bases for calculating Opportunity Cost Matrices}

\author[Y.~Ge]{Yuchen Ge}
\address{Yuchen Ge, Shandong University, Jinan, China}
\email{gycdwwd@163.com}
\urladdr{https://gycdwwd.github.io/}
\thanks{This work was conceived while YG was a research intern in Montreal with financial support from the MITACS Globalink programme. The authors are grateful to MITACS for this support.}

\author[J.~Ortmann]{Janosch Ortmann}
\address{Janosch Ortmann\\ CRM, GERAD and Universit\'e du Qu\'ebec \`a Montr\'eal\\ Case postale 8888, succ. Centre-ville\\ Montr\'eal (QC) H3C 3P8\\ Canada}
\email{ortmann.janosch@uqam.ca}
\urladdr{http://crm.umontreal.ca/~ortmann/}
\thanks{JO would like to thank Rüdiger Schultz for valuable discussions on the subject of the article. He gratefully acknowledges the support provided by an NSERC discovery grant}

\author[W.~Rei]{Walter Rei}
\address{Walter Rei\\ CIRRELT and Universit\'e du Qu\'ebec \`a Montr\'eal\\ Case postale 8888, succ. Centre-ville\\ Montr\'eal (QC) H3C 3P8\\ Canada}
\email{rei.walter@uqam.ca}
\thanks{While working on this paper, the third author was the Canada Research Chair (CRC) in Stochastic Optimization of Transport and Logistics Systems.  In addition, he was supported by the National Sciences and Engineering Research Council of Canada (NSERC) through the discovery grants program. He gratefully acknowledges all the support that is provided by the CRC and NSERC programs.}


\begin{document}
\begin{abstract}
	Opportunity cost matrices, first introduced in \cite{HewittOrtmannRei2022}, are interesting in the context of scenario reduction. We provide new algorithms, based on ideas from algebraic geometry, to efficiently compute the opportunity cost matrix using computational algebraic geometry. We demonstrate the efficacy of our algorithms by computing opportunity cost matrices for two stochastic integer prorams.
\end{abstract}
	
\maketitle
%\tableofcontents

\section{Introduction}

In many real-world applications, decision-making under uncertainty is a challenging problem that requires the consideration of multiple possible scenarios. However, the number of scenarios can grow exponentially with the number of uncertain parameters, making the decision problem intractable. Scenario reduction methods
%\todo{add references: see Julien paper} 
aim to address this issue by reducing the number of scenarios while preserving the key characteristics of the original problem.

Recently \cite{KeutchayanOrtmannRei21, BertsimasMundru2022, HewittOrtmannRei2022} there has been progress in developing problem-driven scenario reduction methods. All of these methods require the computation of a variant of the \emph{opportunity cost matrix} introduced in \cite{HewittOrtmannRei2022}: assign a decision $x^\ast_j$ to each scenario $s_j$. In \cite{KeutchayanOrtmannRei21, HewittOrtmannRei2022}  this is done by choosing an optimal solution to the single-scenario problem. The opportunity cost matrix is computed by evaluating the solution associated to scenario $j$ against scenario $i$. Intuitively, suppose that we hae an oracle that predicts that scenario $j$ will occur. The $(i,j)$ entry of the opportunity cost matrix gives the cost incurred by trusting the oracle and taking a decision under the viewpoint of scenario $j$ if it then turns out that in fact scenario $i$ occurs.

The opportunity cost matrix naturally leads to a decision-based distance on the scenario set: two scenarios $i$ and $j$ can be considered to be close if the $(i,j)$ entry of the opportunity cost matrix is small (and hence the cost of predicting scenario $j$ when actually scenario $i$ occurs leads to a small cost) and far away if the $(i,j)$ entry is large.
 \cite{KeutchayanOrtmannRei21,HewittOrtmannRei2022} propose clustering algorithms based on this distance, which has shown to lead to good approximative solutions and tight bounds on the true objective value.

In this paper, we present a novel approach to efficiently compute opportunity cost matrices using Gröbner and Graver bases. Gröbner \cite{Buchberger1985} and Graver \cite{Graver1975} bases  are algebraic objects that can be used to solve systems of polynomial equations and inequalities. These have been successfully applied to optimisation \cite{Schultz1995,SchultzStougieVanderVlerk1998} 
% \todo{add references} 
and have seen a wide range of applications 
%\todo{add references}.

Our approach leverages the structure of opportunity cost matrices to formulate them as polynomial systems, which can then be solved using Gröbner and Graver bases. The resulting method is fast, accurate, and scales with the number of scenarios and variables.

We demonstrate the effectiveness of our approach through numerical experiments on a well-known integer program from \cite{HemmeckeSchultz2003}. Our results show that the Gröbner basis approach is particularly efficient for a large number of scenarios, whereas the Graver basis approach stays stable with a large number of decision variables.
Our paper contributes to the literature by providing two new approaches to computing opportunity cost matrices. We also prove theoretical results about Graver bases for deterministic and stochastic integer programs.

The rest of the paper is organized as follows. In Section \ref{sec:opp-cost}, we provide background on opportunity cost matrices and their usefulness. In Section \ref{sec:maths}, we introduce the Gröbner and Graver bases and present our mathematical results. Section \ref{sec:algos} describes the algorithms that we then apply to generate our numerical results in Section \ref{sec:results}. Finally, Section \ref{sec:conclusion} concludes.

\subsection*{Notation} Throughout this paper, $\zz$ denotes the set of integers and $\nonnegint$ the subset of non-negative integers.

\section{Opportunity cost matrices}
\label{sec:opp-cost}

Consider the following two-stage stochastic optimisation problem
\begin{align}
	\label{eq:TrueStochOptProblem}
	&	\min_{x\in\x} \ee F(x,\xi) =  \min_{x\in\x} \gamma^\top x + \ee Q(x,\xi) 
	\intertext{where $\gamma\in\rr^d$ and $\x$ is the intersection of the integer lattice with a convex polyhedron. The random part $Q$ of the objective function is given by}
	Q(x,\xi) & =  \min\set{ c_\xi^\top y \colon  T x + A y = h_\xi  \colon y\in\zz^d}
\end{align}
where $c_\xi\in\zz^d $  and $b_\xi\in\zz^k$ for all $\xi$ and $T$ and $A$ are integer matrices of the required dimension. % $\in \zz^{k\times d}$ and $b_\xi\in\zz^k$, $c_\xi\in\zz^d$ 
Observe that we assume $T$ and $A$ to be deterministic matrices, i.e. they are not random.

%In the following, we will sometimes omit the parentheses on $b, c$ and keep in mind that these should be thought of as random quantities, with $b$ also depending on the first stage decision variable $x$.  We will assume throughout that the matrix $A$ is deterministic (not random). 
In general, it is not obvious how to obtain an explicit solution to \eqref{eq:TrueStochOptProblem}. The scenario approach in stochastic programming \cite{BirgeLouveaux} consists of approximating the probability measure $\pp$ on $\xi$ that gives rise to the expectation operator $\ee$ by a discrete measure assigning probabilities\footnote{i.e. $p_n\geq 0$ for all $n$ and $\sum_{n=1}^N p_n=1$} $p_1,...,p_N$ on a finite set of outcomes $\xi_1,...,\xi_N$(often called \emph{scenarios}) for $\xi$. The set $\s=\set{\xi_1,\ldots,\xi_N}$ is then referred to as the \emph{scenario set}.

For example, this can be achieved by the \emph{sample average approximation} \cite{KleywegtShapiroHdM2002}, where the $\xi_i$ are independent samples from $\pp$ and are each assigned the same probability $\nicefrac1N$ of occurring. In any case, the scenario approach yields optimisation problems of the form
\begin{align}
	\label{eq:SIP}
	\min \set{\gamma ^{\top} x+\sum_{i=1}^{N} p_{i} c_{i}^{\top} y_{i}: x\in\x, A y_{i}=h_{i}-T x, y_{i}  \in \mathbb{Z}_{+}^{n}\ \forall\, i}.
\end{align}
In order to properly model the incertainty represented by $\pp$, often a large number of scenarios must be generated. On the other hand solving problems of the form \ref{eq:SIP} when $N$ is large can be very computationally challenging. This motivates the idea of \emph{scenario reduction}: can one find a small subset $\wh\s$ of the scenario set $\s=\set{\xi_1,\ldots,\xi_N}$ such that replacing $\s$ by $\wh\s$ in \eqref{eq:SIP} only incurs a small approximation error. 

Generally speaking, scenario reduction methods can be separated into \emph{distribution-based} and  \emph{problem-based} methods. The former seek to minimise the difference (defined in some suitable sense) between the empirical measure of the subset and the original scenario set. See for example \cite{DupacovaGKRomisch2003, Romisch-conference2009, HenrionKuchlerRomisch2009} for examples of this approach. On the other hand, problem-based methods \cite{BertsimasMundru2022, KeutchayanOrtmannRei21, HewittOrtmannRei2022} look to minimise the difference between the solutions associated to the larger and smaller problems, in other words scenarios are considered to be similar if they lead to similar decisions according to \eqref{eq:SIP}. In \cite{HewittOrtmannRei2022}, the opportunity cost matrix was introduced. 
%
%
%
%
%%\min\set{ c(\xi)\cdot y \colon y \in \nn^d, A y =b(x,\xi)}
%
%In stochastic optimisation, the underlying probability measure 
%\begin{align*}
%	z^\ast & = \min_{x\in \x} \frac1N \sum_{j=1}^N F\rb{x,\xi_j}.
%	\intertext{This motivates the idea of scenario reduction: can one find a subset $\wh\s$ of the scenario set $\s$ such that }
%	\wh z & = \min_{x\in \x} \frac1N \sum_{\xi\in\wh {\s}} F(x,\xi)
%\end{align*}
%is a good approximation for $z^\ast$.

\begin{defn}
	\label{def:opp-cost-mx}
	Consider a set of feasible decisions $x_1,...,x_N\in \x$ such that decision $x_j$ is associated to scenario $s_j$. The \emph{opportunity cost matrix} is the $N\times N$ matrix \a\ whose $(i,j)$ entry is $\a_{ij}=F(x_i,\xi_j)$.
\end{defn}

Clearly there are many interesting choices for the solutions $x_j$. 

\begin{example}
In \cite{HewittOrtmannRei2022} and \cite{KeutchayanOrtmannRei21}, the decisions $x_n$ were chosen to be the single-scenario solutions of \eqref{eq:TrueStochOptProblem}, that is
\begin{align}
	x_j\in\argmin F(x,\xi_j),
\end{align}
\end{example}

In this paper, we are looking for efficient methods of computing the matrix $\a$ in the context of the stochastic linear two-stage problem defined in \eqref{eq:TrueStochOptProblem}.
Since the functions $Q$ and $F$ only differ by the linear term $\gamma\cdot x$, we turn our attention to efficiently computing 
\begin{align}
	\label{eq:OppCostQ}
	Q\rb{x_i,\xi_j} & = \min\set{c_j^\top y \colon A y = b_{ij}},
\end{align}
where $c_j=c(\xi_j)$, $b_{ij}=b(x_i,\xi_j)$. Here, we write $ b(x,\xi)=h_\xi-T_\xi x$.
Since we had assumed $A$ to be deterministic deterministic, the feasible region on the right-hand side of \eqref{eq:OppCostQ} changes only through $b_{ij}$ as $i$ and $j$ vary.

\section{Gröbner and Graver bases}

\label{sec:maths}

In this section we develop the  theory of Gröbner and Graver bases that we require in order to justify and motivate the algorithms that will be described in Section \ref{sec:algos}.

In the previous section we saw that in order to compute the opportunity cost matrix, we are looking to solve integer optimisation problems of the form% \eqref{eq:OppCostQ} 
%First observation: test sets are useful to solve integer optimisation problems of the form
\begin{align}
	\label{eq:generalIP}
	\ip cb: \quad \min \set{c^{\top} z: A z=b, z \in \nn^{d}}
\end{align}
for a given integer matrix $A$ but with varying cost vectors $c\in\mcc$ and constraint vectors $b\in\b$. Here, $\mcc$ and $\b$ are finite sets corresponding to the values that the $c(\xi_m)$ and $b(x,\xi_m)$ of \eqref{eq:OppCostQ} can take. Since we consider the matrix $A$ to be fixed throughout, we do not include it in the notation.

Note that the optimal solution to \eqref{eq:generalIP} may not be unique, which is sometimes inconvenient. Whenever required, we can therefore replace \eqref{eq:generalIP} by its refinement \ip{>_c}b\ that we will define now. We begin my defining a total order\footnote{A \emph{total order} on a set $S$ is a binary relation $\prec$ on $S$ such that $a\prec b$ or $b\prec a$ for all $a,b\in S$ with $a\ne b$.}  on $\nonnegint^d$.

\begin{defn}
	\label{def:OrderC}
	Let $>$ be any total order on $\nonnegint^n$ and fix $c\in\zz^d$. We define another order relation on $\nn^n$ by saying that $x>_c$ y  and only if
	\begin{enumerate}
		\item $c\cdot x>c\cdot y$ or
		\item 	$c\cdot x=c\cdot y$ and $x>y$.
	\end{enumerate}
\end{defn}

Unless otherwise specified, the total order $\prec$ in Definition \ref{def:OrderC} will be the lexicographical order, under which $x>y$ if and only if the leftmost non-zero entry of $x-y$ is positive.
%Note that $>_c$ is a total order on $N^{n}$ but that it may not be a monomial order. \todo{define - appendix?}

\begin{defn}
	A  set $\t_{c} \subseteq \zz^{d}$  is called a \emph{test set} for the family of problems  $\set{\ip cb\colon b \in \mathbb{R}^{l}}$ if
	\begin{enumerate}
		\item $c\cdot t>0$  for all  $t \in \mathcal{T}_{c}$
		\item for every  $b \in \mathbb{R}^{l}$  and for every non-optimal feasible solution  $z_{0} \in \mathbb{Z}_{+}^{d}$  to  $A z=b$ , there exists a vector  $t \in \mathcal{T}_{c}$  such that  $z_{0}-t$  is feasible. Such a vector is called an \emph{augmentation vector}.
	\end{enumerate}
\end{defn}

Given a test set  $\mathcal{T}_{c}$, the \emph{augmentation algorithm} (Algorithm \ref{algo:augmentation}) provides a way to compute the optimum of $\ip {>c}b$ for any $b\in \b$. Intuitively, while  there is  $t \in \mathcal{T}_{c}$  with  $c^{\top} t>0$  such that  $z_{0}-t$  is feasible, we repetitively iterate $z_0$ by assigning $z_{0}:=z_{0}-t$.
Thus, we are looking to efficiently construct a test set $\t_c$ for each $c\in\mcc$. Observe that $\t_c\subseteq \ker(A)$ since both $z_0$ and $z_0-t$ are feasible for any $t\in\t_c$ which implies that
\begin{align}
	A t & = A(z_0 - (z_0-t)) = Az_0 - A(z_0-t) = b-b = 0.
\end{align}

In the following we will exhibit two approaches to constructing such a test set. The first uses Gröbner bases and is described in Section \ref{subsec:Grobner}. The second, based on Graver bases, goes one step further and constructs a \emph{universal test set}, that is, a set $\t$ containing a test set $\t_c$ for each cost vector $c\in\mcc$.

\subsection{The Gröbner Basis Approach}
\label{subsec:Grobner}

It is now well known \cite{Schultz1995,SchultzStougieVanderVlerk1998, BertsimasPerakisTayur2000, Thomas-handbook1998}
% \todo{more citations} 
that Gröbner bases \cite{Buchberger1985, CoxLittleOShea} yield test sets. The \emph{Buchberger algorithm} \cite{Buchberger1985,ContiTraverso1991} takes as input a set of polynomials $f_1,\ldots, f_M$ and returns the reduced Gröbner basis of the ideal generated by the $f_j$. A geometric version of the Buchberger algorithm was given in \cite{Thomas1997}. 
This algorithm is at its most efficient if the generating set is small. We achieve this by using the toric ideal and in particular the methodology of \cite{DiBiaseUrbanke1995}. The advantage is that this can be done once for all cost vectors, which is efficient in the context of computing opportunity cost matrices.

	We will require the following result, which is a direct consequence of Dickson's Lemma \cite{CoxLittleOShea} and can be found as Lemma 2.1.4 of \cite{Thomas1997}.
	
	\begin{lem}
		\label{lem:nonOptimalAlpha}
		Let $A\in\zz^{l\times d}$ and let $\n_{c,b}$ denote the set of non-optimal feasible solutions of $\ip cb$.
		Then there exist $\alpha(t),\ldots,\alpha(t)\in\nn^d$ such that
		\begin{align}
			\bigcup_{b \in \rr^{l}} \n_{c,b}=\bigcup_{i=1}^{t}(\alpha(i)+N^{n})
		\end{align}
	\end{lem}

%
%	
%	\todo{Add some illustrative text and a picture.}
	
	The vectors $\alpha(1),\ldots,\alpha(t)$ allow us to compute a test set for the family of integer programs $\set{\ip cb\colon b\in\b}$, via the the following result, which follows from the Gordon-Dickson Lemma and is stated as Corollary 2.1.10 of \cite{Thomas1997}. See also Lemma 2.1 in \cite{Sturmfels2004}.
	
	\begin{thm}
		\label{thm:ThomasTestSet}
		Fix $c\in\zz^d$, let $\alpha(1),\ldots,\alpha(t)$ as in Lemma \ref{lem:nonOptimalAlpha} and let $\beta(i)$ be the (unique) optimum of the program $\ip{>_c}{A \alpha (i)}$. Then the set
		\begin{align}
			\mathcal{G}_{A,c}=\set{(\alpha(i)-\beta(i)),i=1,2,...,t}
		\end{align}
		is a minimal test set for the family of problems $\set{\ip cb\colon b\in\rr^l}$.
	\end{thm}

	\begin{defn}
		\label{def:GB-matrix}
	The set  $\mathcal{G}_{A,c}$ is called the \emph{Gröbner basis} for the matrix $A$. 
\end{defn}

	While this result gives a formula for the desired test set, it requires computation of the $\alpha(i)$ from Lemma \ref{lem:nonOptimalAlpha} and, crucially, of the $\beta(i)$, which are given in terms of solutions of integer problems itself. A geometric construction of the test set was given in \cite{Thomas1997}.  
	
	As explained above, we need to compute the solutions to $\ip {>_c}b$ as $c$ varies. If we would like to take the Gröbner basis approach, we would need to recompute the Gröbner basis for every cost vector $c\in\mcc$. Therefore, we proceed in a different manner, based on ideas from the theory of polynomial rings.% For a more detailed account of the theory behind this approach see Appendix \ref{appendix:polynomial} and the references given there.
	
	Let $R=\polring{x}{n}$ denote the ring of polynomials in $n$ variables. We will make frequent use of a close relation between vectors of integers and differences of monomials in $R$. Any $n$-vector of non-negative integers $\alpha=(\alpha_1,\ldots,\alpha_n)\in\nonnegint^n$ can be identified with a monomial $x^\alpha\in R$ defined by 
	\begin{align}
		x^\alpha = x_1^{\alpha_1} \cdots x_n^{\alpha_n} = \prod_{j=1}^{n} x_j^{\alpha_j},
	\end{align}
	and this mapping can be extended to integer-valued vectors as follows. For any  $\alpha=(\alpha_1,\ldots, \alpha_n)\zz^n$, define $\alpha^+=(\alpha_1^+,\ldots,\alpha_n^+)$ and $\alpha^-=(\alpha_1-,\ldots,\alpha_n^-)$, recalling that the positive part of a real number $x$ is defined to be $x^+=\max(x,0)$ and the negative part of $x$ is $x^-=(-x)^+$. Then both $\alpha^+, \alpha^-\in\nonnegint^n$ and $\alpha=\alpha^+-\alpha^-$.  We now define a mapping
	\begin{align}
		\label{eq:defPhi}
		\phi\colon\zz^n\longrightarrow R \quad\text{by}\quad \phi(\alpha) = x^{\alpha^+}- x^{\alpha^-}.
	\end{align} 

	The notion of a Gröbner basis associated to a matrix from Definition \ref{def:GB-matrix} is actually a special case of a more general concept, namely that of a Gröbner basis of a polynomial ideal.  In the following, we give a very brief account of the definition and basic properties of Gröbner bases. For more details see for example Chapter 1 of \cite{Sturmfels-AMS}. %We begin by defining ideals of a ring. An ideal is an additive subgroup
	
	\begin{defn}
		\label{def:ideal}
		A subset $I\subseteq R$ is said to be an \emph{ideal} of $R$ if it satisfies the following two properties:
		\begin{enumerate}[a)]
			\item For any $x,y\in I$ we have $x-y\in I$ (that is, $I$ is an additive subgroup of $R$),
			\item For any $x\in I$ and $r\in R$ we have $rx \in I$.
		\end{enumerate}
		\end{defn}
	
	\begin{example}
		\label{ex:idealGenerated}
		Given polynomials $f_1,\ldots,f_n$, the \emph{ideal generated by} $f_1,\ldots,f_n$ is the set of all $R$-linear combinations of the $f_j$:
		\begin{align}
			\label{eq:idealGenerated}
			\left\langle f_1,\ldots,f_n\right\rangle & = \set{ \sum_{j=1}^n r_j f_j \colon r_1,\ldots,r_n \in R }.
		\end{align}
	\end{example}
	
	\begin{example}
		\label{ex:toric}
		The \emph{toric ideal} $I_A$ of a matrix $A\in\zz^{d\times n}$ is generated by differences of monomials $x^u-x^v$ such that $u-v$ lies in the kernel of $A$:				
%		 defined as follows. Denote the columns of $A$ by $\bm{a}_1,\ldots,\bm{a}_n$. Consider the map $\pi\colon \nn^n \longrightarrow\zz^d$ defined by
		\begin{align}
			\label{eq:toric}
			I_A & = \left\langle x^u-x^v \colon Au=Av \right\rangle.
		\end{align}
%		The \emph{toric ideal} \cite[Lemma 4.1]{Sturmfels-AMS} is generated by differences of 
	\end{example}
%Write out definition as in Sturmfels and then state \eqref{eq:toric} as a property?
See Section 4 of \cite{Sturmfels-AMS} for the theory of toric ideals. A polynomial $f\in R$ consists of monomials $x^\alpha$ with \emph{degree} $\alpha\in\nonnegint^n$. Throughout, we fix a total order (recall Definition \ref{def:OrderC} and the discussion surrounding it) $\prec$ on $\nn^n$. In principle, any total order will do, but for us, $\prec$ will always be the total order $<_c$ defined in Definition \ref{def:OrderC}. Given $\prec$, we can compare the different monomials by comparing their degrees. This allows us to define the \emph{initial monomial} of $f$ to be that with the greatest degree (with respect to $\prec$). The initial monomial of $f$ is denoted $\init\prec f$. Given an ideal $I$ of $R$, the \emph{initial ideal} of $I$ is the  ideal generated by the initial monomials of the elements of $I$:
\begin{align}
	\init\prec f & = \left\langle \init\prec f \colon f\in I \right\rangle.
\end{align}
	
	We are now in a position to define Gröbner bases for polynomial ideals. 
	
	\begin{defn}
		\label{def:GB-ideal}
		Let $I$ be an ideal of $R$. A subset $\g$ of $I$ is said to be a \emph{Gröbner basis} of $I$ if the inital ideal of $I$ is generated by the inital monomials of the elements of $\g$:
		\begin{align}
			\label{eq:GB-ideal}
			\init\prec I & = \left\langle\init\prec g \colon g\in\g  \right\rangle.
		\end{align}
	A Gröbner basis $\g$ of $I$ is said to be \emph{minimal} if \eqref{eq:GB-ideal} no longer holds whenever  any one element is removed from 	$\g$. Also, $\g$ is said to be \emph{reduced} if for any distinct $g_1,g_2\in\g$ it holds that no term of $g_1$ is a multiple of $\init \prec {g_2}$.
	\end{defn}

Given an ideal $I$ of $R$ and a total order on $\nn^n$, there is only one reduced Gröbner basis $\g$ such that the coefficients of $\init\prec g$ are equal to $1$ for each $g\in\g$. In other words, reduced Gröbner bases are unique up to multiplying the elements by constants. 
	
	\begin{prop}
		%\todo{reference?}
		Recall the vectors $\alpha(i)$ and $\beta(i)$ from \ref{thm:ThomasTestSet}. The set of polynomials
		\begin{align}
			\g_{I_A} &= \set{x^{\alpha(j)}-x^{\beta(j)} \colon j\in\set{1,\ldots,t} }
		\end{align}
	is the reduced Gröbner basis of the toric ideal $I_A$ with respect to the ordering $<_c$.
	\end{prop}

	We conclude that the two notions of Gröbner basis (of a matrix on the one hand and of a polynomial ideal on the other) are actually the same: they are related via the map $\phi$ defined in \eqref{eq:defPhi}. In fact, the original Buchberger algorithm \cite{Buchberger1985} is formulated in this setting.
	
In \cite{DiBiaseUrbanke1995}, an efficient algorithm to compute a small generating set of the toric ideal $I_A$, see Algorithm \ref{algo:genset-toric}. Observe that the toric ideal only depends on the matrix $A$, and not on the cost vector $c$. This motivates our \emph{kernel algorithm}, described in detail in Algorithm \ref{algo:kernel}:

\begin{enumerate}
 \item Given the matrix $A$, compute a small generating set of the toric ideal  $I_{A}$ following \cite{DiBiaseUrbanke1995}.
	\item For each $j\in\set{1,\ldots,N}$, apply Buchberger's algorithm to compute the Gröbner base with respect to $A$ and $<_{c_j}$.
	\item Use the correspondence $ (x^u-x^v) \longleftrightarrow u-v$ from \eqref{eq:defPhi} to obtain the test set with respect to $(A,c_j)$.
	\item For each $i,j\in\set{1,\ldots,N}$, apply the augmentation algorithm (Algorithm \ref{algo:augmentation}) to compute $Q(x_i,\xi_j)$, and hence the $(i,j)$-entry of the opportunity cost matrix.
\end{enumerate}

%\todo{need to specify what is taken from \cite{DiBiaseUrbanke1995} and what is new}.
	
%		New idea: toric ideal approach. Use algorithm of \cite{DiBiaseUrbanke1995} to compute a small generating set for the toric ideal, which only depends on $A$ and therefore only needs to be computed once. This greatly speeds up the Buchberger algorithm that we need to compute for each $c$.
%	
%	The key ingredient to our approach is the \emph{toric ideal} of the matrix $A$ that will be defined below. It is a 
	%that we explain in the following.

%	
%	Toric ideal $I_A$ : defined at bottom of the page 5 of Algsto\_paper. Theorem 3.7 (gathered from other results but possibly new in its current form) tells us how to calculate $I_A$.
%	
	
	\subsection{The Graver basis approach}
	\label{subsec:Graver}

In the previous section, computing the $N\times N$ opportunity cost matrix still required computing $N$ test sets via the Gröbner bases. This motivates the question whether it is possible to construct a single test set that not only covers all right-hand side vectors $b$ but also all cost vectors $c$. 

\begin{defn}
	A set $\t\subseteq $ is an \emph{universal test set} for the matrix $A$ if it contains a test set $\t_c$ for every cost vector $c$.
\end{defn}

Clearly, the Gröbner basis approach alone is insufficient for computing a universal test set, as this would require computing an infinite number of Gröbner bases. Instead, the Graver basis approach will be used. We begin by stating the definition of a Graver basis.

%%%%%%%%%%%%%%%%%%%%%%%%%%%%%%%%%%%%%%%%

\begin{defn}
	Assume $a=(a_1,a_2,...,a_n), b=(b_1,b_2,...,b_n) \in \mathbb{Z}^{n}$. Then $a \sqsubseteq b$ if  $\forall i$:  $a_{i} b_{i} \geq 0 $ and $\left|a_{i}\right| \leq\left|b_{i}\right|$. 
\end{defn}

\begin{defn} Fix a integer matrix $A$. The Graver Basis of $A$ is the set  of  $\sqsubseteq$-minimal elements in  $\operatorname{Ker}_{\mathbb{Z}}(A): =\left\{u \in \mathbb{Z}^{n}: A \cdot u=0,u\neq 0\right\}$, denoted by $\Gamma_A$. 
\end{defn}
It is not immediately obvious that  $\Gamma_A$ is finite. However, finiteness of $\Gamma_A$ follows from the following result gathered from \cite{Hemmecke2003}.

\begin{prop}
	Let $\mathcal{G}_{A, c}$ be the reduced  Gröbner basis of $A$ with respect to $>c$. Then $$ \bigcup_{c \in \mathbb{Z}^{n}} \mathcal{G}_{A, c} \subseteq \mathcal{G} \mathcal{R}_{A} .$$
\end{prop}

Next, we state and prove a result about the relationship between the Graver Basis of the stochastic integer program \ref{eq:SIP} and that of the deterministic integer program obtained by choosing a single scenario and assigning probability $1$ to it:
\begin{equation}
	\label{eq:SIP-onescen}
	\min \left\{c^{\top} x+ p q^{\top} y_{}: A x = b, x  \in \mathbb{Z}_{+}^{m}, T x+W_{} y_{}=h_{}, y_{}  \in \mathbb{Z}_{+}^{n} \right\}
\end{equation}
Then we assert a new result, which connects the Graver Basis of the SIP problem and the Graver Basis of the IP problem:

\begin{thm} Denote the Graver Basis of the SIP problem (5) by $\Gamma_{N}$, and that of the IP problem \eqref{eq:SIP-onescen} by $\Gamma_{1}$  for  $i=1,2,...,N$. If $\operatorname{Ker}_{\mathbb{R}}(A)=0$, we have the following relationship $$\Gamma_{N}=\{(0,0,...,v_i,...,0):(0,v_i)\in \Gamma_{1},i=1,2,...,n\}$$\end{thm}

\begin{proof}
	The corresponding coefficient matrix of \eqref{eq:SIP} is of the form
	$$
	A_N:=\left(\begin{array}{ccccc}
		A & 0 & 0 & \cdots & 0 \\
		T & W & 0 & \cdots & 0 \\
		T & 0 & W & \cdots & 0 \\
		\vdots & \vdots & \vdots & \ddots & \vdots \\
		T & 0 & 0 & \cdots & W
	\end{array}\right)
	$$
	and  the corresponding coefficient matrix of (6) is of the form 
	$$
	A_{1}:=\left(\begin{array}{cc}
		A & 0  \\
		T & W  \\
	\end{array}\right).
	$$
	Since $\operatorname{Ker}_{\mathbb{R}}(A)=0$, we have $\operatorname{Ker}_{\mathbb{Z}}(A_N)=(0,v_1,v_2,...,v_n):(0,v_i)\in \operatorname{Ker}_{\mathbb{Z}}(A_1) \}$. From the definition of the Graver Basis, we can arrive at the conclusion.
\end{proof}

	\section{Algorithms}

\label{sec:algos}

	In this section, we describe numerical algorithms that follow from the mathematical development of the previous section. We first introduce the augmentation algorithm that computes an optimal solutions to $\ip cb$ from a test set $\t_c$ and a feasible solution $z_0$.

\begin{algorithm}[!ht]
	\caption{The augmentation algorithm}\label{algo:augmentation}
	\begin{algorithmic}
		\Require Cost vector $c$, right-hand side vector $b$
		\Require Matrix $A$
		\Require Test set $\t_c$
		\Require $z_0$ such that $Az_0=b$.
		\While{$\exists \, t\in \t_c$ such that $z_0-t$ is feasible}
		\State $z_0 \gets z_0-t$
		\EndWhile
		
		\Return Optimal solution $z_0=\argmin\set{c^T x\colon  Ax=b}$
	\end{algorithmic}
\end{algorithm}

%to solve integer problems $\ip cb$ for varying $c,b$.

\begin{algorithm}[!h]
	\caption{Algorithm to compute a small generating set of the Toric ideal \cite{DiBiaseUrbanke1995}}
	\label{algo:genset-toric}		
	\begin{algorithmic}
		\Require Integer matrix $A$
		\State Calculate a $\mathbb{Z}-$basis $K$ for $\operatorname{ker} (\pi_{*})$
		%			\\
		%			\noindent 1.  $ (to simplify notation, we will use  $K$  to denote a basis for  $\operatorname{ker} \pi_{*}$  the matrix in  $\mathbb{Z}^{k \times l}$  whose rows are the vectors in $\operatorname{ker} \pi_{*}$ ) \\
		\State Find an equivalent basis  $K^{\prime}$  such that all rows of  $K^{\prime}$  lie in the same orthant.
		\State  $J \gets$  index set of all columns with negative entries and let  $K_{J}^{\prime}$  be the matrix obtained from  $K^{\prime}$  by reversing the signs of the columns indexed by  $J$. 
		\State $G_{J}\gets \varphi\left(K_{J}^{\prime}\right)$.
		\While{$J\ne \varnothing$}
		\State Take  $j \in J$  and let  $G_{J \setminus\set{j}}$  be the result of $T_{j}$ operating on the reduced Gröbner basis for  $\langle G_{J} \rangle$ with respect to a term order that eliminates  $x_{j}$ 
		\State $J \gets   J \setminus\set{j}$
		\EndWhile
		
		\Return The generating set $G_{\varnothing}$ of $I_A$.
	\end{algorithmic}
\end{algorithm}

For details about Algorithm \ref{algo:genset-toric} see \cite{DiBiaseUrbanke1995}.

The kernel algorithm takes as input the matrix $A$, a sequence of cost vectors $c$ and an array of right-hand side vectors $b$ and computes the corresponding opportunity cost matrix, by first applying Algorithm \ref{algo:genset-toric} and then repeatedly the Buchberger algorithm. In this way, a Gröbner basis for $A$ is computed with respect to each $>_c$. The augmentation algorithm \ref{algo:augmentation} then gives the entries of the opportunity cost matrix for every $c$ and every $b$.

\begin{algorithm}[!h]
	\caption{The Kernel Algorithm}
	\label{algo:kernel}
	\begin{algorithmic}
		\Require Cost vectors $\set{c_1,\ldots,c_N}$
		\Require Right-hand side vectors $\set{b_{i,j}}_{i,j=1}^N$
		\Require Matrix $A$
		\State Compute a generating set $\set{f_1,\ldots,f_r}$ of the Toric ideal using Algorithm \ref{algo:genset-toric}
		\State $i\gets 1$
		\While{$i\leq N$}
		\State Compute the Gröbner basis $\g_i$ of $(A,>_{c_i})$ via the Buchberger algorithm 
		\State $j\gets 1$
		\While{$j\leq N$}
		\State Compute $\a_{i,j}=\ip{c_i}{b_{i,j}}$ by inputting $\g_i$ into the Augmentation Algorithm \ref{algo:augmentation}
		\State $j\gets j+1$
		\EndWhile
		%		\textbf{TODO: write this properly}
		%		$\underline{\text{Input:}}$ integer optimisation problems $(I P)_{c, b}$  with $b,c$ varying and $A$ fixed\\
		%		$\underline{\text {Output:}} $ the test sets corresponding to all these $(I P)_{c, b}$ 
		%		\begin{enumerate}[1.]
			%			
			%			\item Apply Buchberger's algorithm to compute the Gröbner Bases of each $(I P)_{c, b}$
			\State $i\gets i+1$
			\EndWhile
			%			\item Use the correspondence $ (x^{\alpha_i}-x^{\beta_i}) \leftrightarrow (\alpha_i-\beta_i)$ to get all the test sets.
			%		\end{enumerate}
		
		\Return The opportunity cost matrix $\rb{\a_{i,j}}_{i,j=1}^N$	
	\end{algorithmic}
\end{algorithm}

For the Graver basis approach, we apply the modified augmentation algorithm \ref{algo:augmentation-graver} for a family of integer problems of the form $\min\set{c^\top x\colon Ax=b}$ as both $b$ and $c$ vary.

\begin{algorithm} 
	\caption{Graver Basis method}
	\label{algo:augmentation-graver}
	\begin{algorithmic}
		\Require Cost vectors $\set{c_1,\ldots,c_N}$
		\Require Right-hand side vectors $\set{b_{i,j}}_{i,j=1}^N$
		\State Compute the Graver Basis $\Gamma_A$   of  $A$
		\Require 
		\State $i\gets 1$
		\While{$i\leq N$}
		\State $j\gets 1$
		\While{$j\leq N$}
		\State Compute a feasible solution $z_{i,j}$ to $Ax=b_{i,j}$
		\While{$\exists\, t \in \Gamma_A\colon c_i^{\top} t>0$ and s.t. that $z_{0}-t$ is feasible}
		\State  $z_{i,j}:=z_{i,j}-t$
		\EndWhile
		\State Compute $\a_{i,j}=c_j^\top  z_{i,j}$
		\EndWhile
		\EndWhile
		
		\Return 		The opportunity cost matrix $\a$ 
	\end{algorithmic}
\end{algorithm}

\ \\

\vfill\eject

\section{Preliminary results}
\label{sec:results}

In this section we test our approach on some numerical examples. The first is a well-known purely combinatorial stochastic program  introduced in \cite{HemmeckeSchultz2003}. The second is a stochastic network design problem. 

\subsection{The example of \cite{HemmeckeSchultz2003}}
\label{subsec:HS2003}

In  \cite{HemmeckeSchultz2003}, the following stochastic integer program was introduced and studied.

\begin{align}
	\label{eq:HemmeckeSchultz-original}
	\min & \left\{35 x_{1}+40 x_{2}+\frac{1}{N} \sum_{\nu=1}^{N} 16 y_{1}^{\nu}+19 y_{2}^{\nu}+47 y_{3}^{\nu}+54 y_{4}^{\nu}\right\} \\
	\notag
	\text{s.t. }  & x_{1}+y_{1}^{\nu}+y_{3}^{\nu}  \geq \xi_{1}^{\nu} \\
	\notag
	& x_{2}+y_{2}^{\nu}+y_{4}^{\nu}  \geq \xi_{2}^{\nu} \\
	\notag
	& 2 y_{1}^{\nu}+y_{2}^{\nu}  \leq \xi_{3}^{\nu} \\
	\notag
	& y_{1}^{\nu}+2 y_{2}^{\nu}  \leq \xi_{4}^{\nu}, \\
	\notag
	& x_{1}, x_{2}, y_{1}^{\nu}, y_{2}^{\nu}, y_{3}^{\nu}, y_{4}^{\nu}  \in \mathbb{Z}_{+}
\end{align}

As in \cite{HemmeckeSchultz2003} the scenarios $\xi^\nu$ are chosen uniformly from the four-dimensional squares $ [300,12000] \times [300,12000]\times [200,12000] \times[200,12000]$. 
In order to bring the problem in the form \eqref{eq:TrueStochOptProblem}, we introduce slack variables $u_j^\nu$ for $j\in\set{1,2,3,4}$ and obtain

\begin{align}
	\label{eq:HemmeckeSchultz-reformulated}
	\min & \left\{\frac{1}{N} \sum_{\nu=1}^{N} 35 x_{1}+40 x_{2}+ 16 y_{1}^{\nu}+19 y_{2}^{\nu}+47 y_{3}^{\nu}+54 y_{4}^{\nu}\right\} \\
	\notag
	\text{s.t. }  & x_{1}+y_{1}^{\nu}+y_{3}^{\nu} -u_1^{\nu} = \xi_{1}^{\nu} \\
	\notag
	& x_{2}+y_{2}^{\nu}+y_{4}^{\nu} -u_2^{\nu}  = \xi_{2}^{\nu} \\
	\notag
	& 2 y_{1}^{\nu}+y_{2}^{\nu} +u_3^{\nu} = \xi_{3}^{\nu} \\
	\notag
	& y_{1}^{\nu}+2 y_{2}^{\nu}  + u_4^{\nu} =  \xi_{4}^{\nu}, \\
	\notag
	& x_{1}, x_{2}, y_{1}^{\nu}, y_{2}^{\nu}, y_{3}^{\nu}, y_{4}^{\nu}, u_{1}^{\nu}, u_{2}^{\nu}, u_{3}^{\nu}, u_{4}^{\nu} \in \mathbb{Z}_{+}.
\end{align}

In order to test our approach, we calculated the seconds of CPU time used by the programs on the Apple Silicon M1 chip. We observe that it is easy to read off a  feasible solution to the problem \eqref{eq:HemmeckeSchultz-reformulated}
\begin{align}
	\label{eq:feasible-HS}
	x_{1}=x_{2}=y_{1}^{\nu}=y_{2}^{\nu}=0, y_{3}^{\nu}=\xi_{1}^{\nu}, y_{4}^{\nu}=\xi_{2}^{\nu}\quad (\nu=1, \ldots N)
\end{align}

%%%%%%%%%%%%%%%%%%%%%%%%%%%%%%%%%%%%%%%%
%%%%%%%%%%%%%%%%%%%%%%%%%%%%%%%%%%%%%%%%
%%%%%%%%%%%%%%%%%%%%%%%%%%%%%%%%%%%%%%%%
%%%%%%%%%%%%%%%%%%%%%%%%%%%%%%%%%%%%%%%%

The task is to compare the computational costs of the opportunity-cost matrix in the SIP problem in the last section. In other words, we compute a family of integer programming problems with three algorithms. We test the program's running time in CPU seconds on an Apple Silicon M1 chip with Python's $\texttt{time}$ package. Also, we introduce the modern MINLP solver, specially designed for mixed-integer programming problems, implemented by Python's $\texttt{gekko}$ package. We shall compare the computational costs of  our algorithms \ref{algo:kernel} and \ref{algo:augmentation-graver} with the MINLP solver in two cases. One is when we fix the sub-problem, and the other is when we fix the number of scenarios.

We shall compute the computational time when the quantity of scenarios equals 1, 20, 40, 60, 80, 100, 120, 140, 160, 180, and 200. We gather the numerical results in the table below.

\begin{table}[htbp]
	\centering
	\caption{Comparison between different solvers ( Timing units: seconds )}
	\begin{tabular}{|c|c|c|c|}
		\hline
		Scenario Number  & Kernel Method & MINLP solver & Graver Method \\
		\hline
		1 & 0.20098  & 11.33629  & 131.06186   \\
		20 & 1.21699  & 219.18796 & 131.66931   \\
		40 & 2.33655  & 439.71981 & 132.02384   \\
		60 & 3.53545  & 643.3165  & 132.35069   \\
		80 & 4.68832  & 994.41339     & 132.72208  \\
		100 & 5.93403  & 1250.31321   & 133.14721   \\
		120 & 7.24906  & 1508.7674  & 133.61075  \\
		140 & 8.56163  & 1803.57992     & 134.06991  \\
		160 & 10.08377 & 2075.8631  & 134.61703   \\
		180 & 11.25169 & 2333.33849 & 135.21214  \\
		200 & 12.88583 & 2584.11191 & 135.8798  \\
		\hline
	\end{tabular}
\end{table}

It is remarkable that it only took about 0.1 seconds to compute the generating set of the toric ideal of $A$, essential for the Gröbner Bases computation in all scenarios. By contrast, it took approximately 130 seconds to compute the Graver Basis of $A$. Also we visualize the table in a line chart below. The horizontal axis represents the increasing quantity of scenarios, and the vertical axis represents the running time. The three curves with different colors in the picture represent three different methods.

\begin{figure}[h]
	\centering
	\includegraphics[scale=0.5]{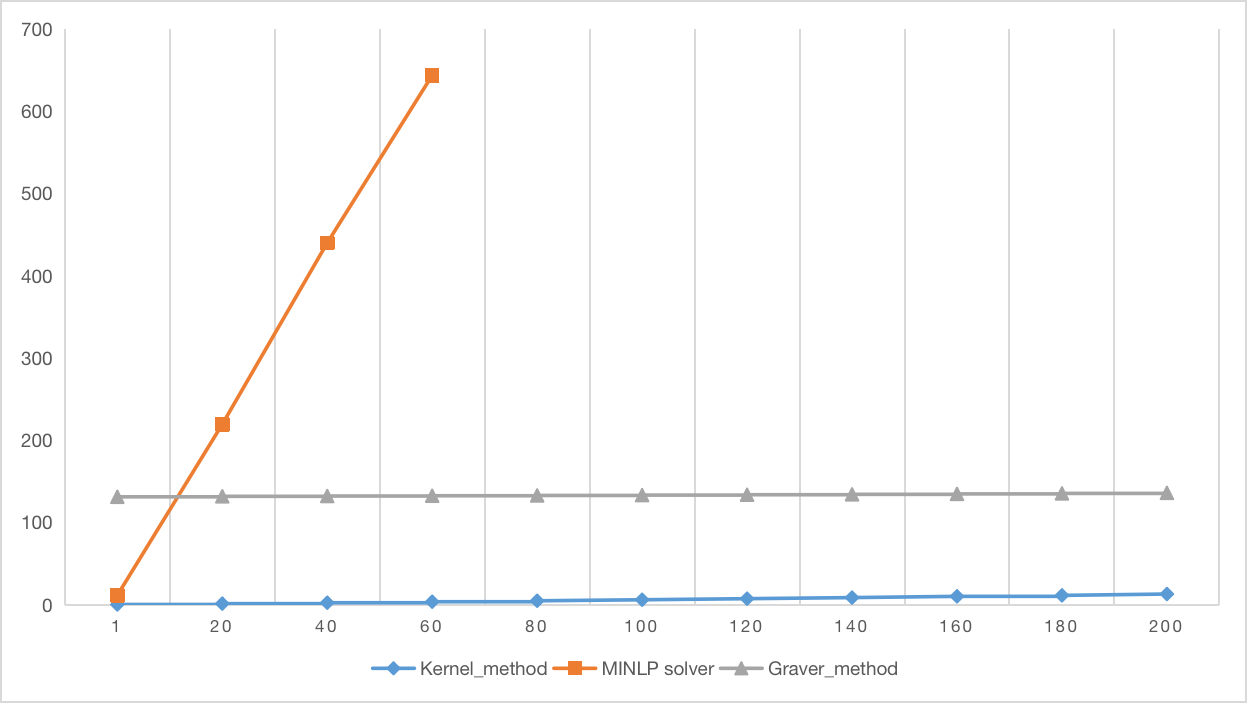}
	\caption{Trends of Computational Costs when the Sub-problem is Fixed}
	\label{fig:scen-var-variables-fixed}
\end{figure}

Figure \ref{fig:scen-var-variables-fixed} shows that Algorithm  \ref{algo:augmentation-graver} took almost 130 seconds to initialize and then the curve rises slowly. The computational cost of the MINLP solver grows extremely fast and starts to overpower the Algorithm   \ref{algo:augmentation-graver} when $N=20$. Generally, the curve of Algorithm \ref{algo:kernel} shows superior performance when compared with the other two algorithms.
%
%Algorithm \ref{algo:augmentation-graver} has its unique advantage: it calculates the Graver Basis once and for all. However, the computational cost might dramatically boost when the matrix $A$ is large. The proposed algorithm 3.1 fixes the problem by only computing the generating set of the toric ideal of $A$. The computation of the generating set is not as costly and beneficial to the calculation of the Gröbner Basis corresponding to all scenarios. After calculating the generating set, we can get each scenario's Gröbner Basis without effort. So the time of computing the opportunity-cost matrix is significantly reduced.

Next, we compare when the complexity of the sub-problems increases. Below is the graph showing the trends of computational costs when the size of scenarios is fixed to 200 but the number of scenarios is increased. In Figure \ref{fig:scen-fixed-variables-var}, with a fixed 200 scenarios, var*$x$ represents $x$  variables in the particular programming problem. As can be seen, by comparing the curves, the Gröbner Basis approach (the blue curve) is overall superior to the Graver Basis approach (the grey curve) and the traditional MINLP solver (the yellow curve). Also, regarding details, the Graver Basis method begins to suffer from the curse of dimensionality, starting from the problem's complexity being 20 variables, and so is the Gröbner Basis method but more dramatically. However, the MINLP solver does not reflect the apparent complexity when it suffers from the curse. Generally, the MINLP solver applies to situations where the sub-problem is large, and the number of scenarios is small. In contrast, the Gröbner Basis applies to situations where the sub-problem has little complexity, and the number of scenarios is significant.

\begin{figure}[h]
	\centering
	\includegraphics[scale=0.5]{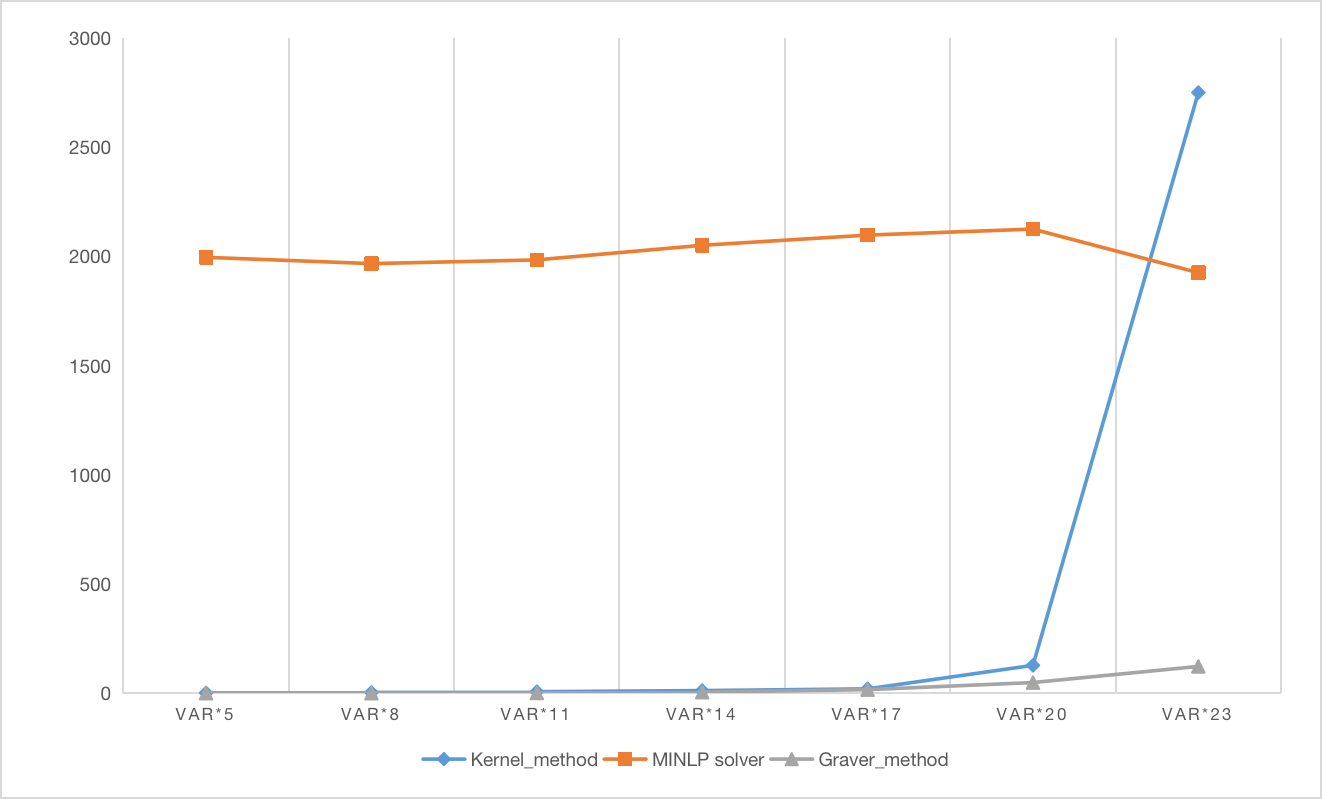}
	\caption{Trends of Computational Costs when the Quantity of Scenarios is Fixed}
	\label{fig:scen-fixed-variables-var}
\end{figure}

We have yet to reflect one thing in the experiments. By testing our algorithm hundreds of times, we have come to the empirical conclusion that the time increase of our proposed algorithm in computing Grobner Basis will be dramatic for $N > 15$. And there will be geometric growth for the computational time regarding Graver Basis when the complexity is between 10 and 15 variables. Also, surprisingly, we can no longer compute the Gröbner Basis after 15 variables' complexity. We refer to this phenomenon as the curse of dimensionality in computational algebraic geometry regarding SIP problems. We conjecture that the crux is that the number of ideal minimal generating elements will increase dramatically as the dimension increases, thus making Bunchberger's algorithm computationally costly.

Finally, the crux of the problem, that two algebraic methods don't function when the variables' number exceeds 15, is the computation of the Gröbner Basis of some specific ideals, specifically the toric ideal of the coefficient matrix $A$, unavoidable in each of the two algorithms. We have tried many applications specialized in algebraic calculations, for example, SageMath, Macaulay2, etc. Still, when the variables' number exceeds 15, usually the exponential of the generating elements of the polynomial ideal, which is required in the algorithm, is too great to be computable. For example, the computation of the Gröbner Basis of the matrix $A=(B, I)$ where $I$ denotes the identity matrix and 
$$
B = \begin{pmatrix}
	1 & 1 & 1 & 1 & 1  & 1  & 1  & 1  & 1  & 1  \\
	1 & 2 & 3 & 4 & 5  & 6  & 7  & 8  & 9  & 10 \\
	2 & 3 & 4 & 5 & 6  & 7  & 8  & 9  & 10 & 11 \\
	3 & 4 & 5 & 6 & 7  & 8  & 9  & 10 & 11 & 12 \\
	4 & 5 & 6 & 7 & 8  & 9  & 10 & 11 & 12 & 13 \\
	5 & 6 & 7 & 8 & 9  & 10 & 11 & 12 & 13 & 14 \\
	6 & 7 & 8 & 9 & 10 & 11 & 12 & 13 & 14 & 15 \\
\end{pmatrix}
$$

involves the computation of the Gröbner Basis of the toric ideal below\\

\noindent $\text{ideal}(x_{8}^2-x_{7}*x_{9}, x_{7}*x_{8}-x_{6}*x_{9}, x_{6}*x_{8}-x_{5}*x_{9}, x_{5}*x_{8}-x_{4}*x_{9}, x_{4}*x_{8}-x_{3}*x_{9}, x_{3}*x_{8}-x_{2}*x_{9}, x_{2}*x_{8}-x_{1}*x_{9}, x_{1}*x_{8}-x_{0}*x_{9},  x_{7}^2-x_{5}*x_{9}, x_{6}*x_{7}-x_{4}*x_{9}, x_{5}*x_{7}-x_{3}*x_{9}, x_{4}*x_{7}-x_{2}*x_{9}, x_{3}*x_{7}-x_{1}*x_{9}, x_{2}*x_{7}-x_{0}*x_{9}, x_{1}*x_{7}-x_{0}*x_{8}, x_{6}^2-x_{3}*x_{9}, x_{5}*x_{6}-x_{2}*x_{9}, x_{4}*x_{6}-x_{1}*x_{9}, x_{3}*x_{6}-x_{0}*x_{9}, x_{2}*x_{6}-x_{0}*x_{8}, x_{1}*x_{6}-x_{0}*x_{7},  x_{5}^2-x_{1}*x_{9}, x_{4}*x_{5}-x_{0}*x_{9}, x_{3}*x_{5}-x_{0}*x_{8}, x_{2}*x_{5}-x_{0}*x_{7}, x_{1}*x_{5}-x_{0}*x_{6}, x_{4}^2-x_{0}*x_{8}, x_{3}*x_{4}-x_{0}*x_{7}, x_{2}*x_{4}-x_{0}*x_{6}, x_{1}*x_{4}-x_{0}*x_{5}, x_{3}^2-x_{0}*x_{6}, x_{2}*x_{3}-x_{0}*x_{5}, x_{1}*x_{3}-x_{0}*x_{4}, x_{2}^2-x_{0}*x_{4}, x_{1}*x_{2}-x_{0}*x_{3}, x_{1}^2-x_{0}*x_{2},  x_{8}*x_{11}*x_{12}*x_{13}*x_{14}*x_{15}*x_{16}-x_{9}, x_{7}*x_{11}*x_{12}*x_{13}*x_{14}*x_{15}*x_{16}-x_{8}, x_{6}*x_{11}*x_{12}*x_{13}*x_{14}*x_{15}*x_{16}-x_{7},
x_{5}*x_{11}*x_{12}*x_{13}*x_{14}*x_{15}*x_{16}-x_{6}, x_{4}*x_{11}*x_{12}*x_{13}*x_{14}*x_{15}*x_{16}-x_{5},
x_{3}*x_{11}*x_{12}*x_{13}*x_{14}*x_{15}*x_{16}-x_{4}, x_{2}*x_{11}*x_{12}*x_{13}*x_{14}*x_{15}*x_{16}-x_{3},
x_{1}*x_{11}*x_{12}*x_{13}*x_{14}*x_{15}*x_{16}-x_{2}, x_{0}*x_{11}*x_{12}*x_{13}*x_{14}*x_{15}*x_{16}-x_{1},
x_{9}*x_{10}*x_{14}*x_{15}^2*x_{16}^3-x_{0}*x_{6}*x_{11}^2* x_{12}, x_{8}*x_{10}*x_{14}*x_{15}^2*x_{16}^3-x_{0}*x_{5}*x_{11}^2* x_{12}, x_{7}*x_{10}*x_{14}*x_{15}^2*x_{16}^3-x_{0}*x_{4}*x_{11}^2* x_{12},
x_{6}*x_{10}*x_{14}*x_{15}^2*x_{16}^3-x_{0}*x_{3}*x_{11}^2* x_{12}, x_{5}*x_{10}*x_{14}*x_{15}^2*x_{16}^3-x_{0}*x_{2}*x_{11}^2* x_{12}, x_{4}*x_{10}*x_{14}*x_{15}^2*x_{16}^3-x_{0}*x_{1}*x_{11}^2* x_{12},
x_{3}*x_{10}*x_{14}*x_{15}^2*x_{16}^3-x_{0}^2*x_{11}^2* x_{12}, x_{0}*x_{5}*x_{11}^3*x_{12}^2*x_{13}-x_{9}*x_{10}*x_{15}*x_{16}^2, x_{0}*x_{4}*x_{11}^3*x_{12}^2*x_{13}-x_{8}*x_{10}*x_{15}*x_{16}^2,
x_{0}*x_{3}*x_{11}^3*x_{12}^2*x_{13}-x_{7}*x_{10}*x_{15}*x_{16}^2, x_{0}*x_{2}*x_{11}^3*x_{12}^2*x_{13}-x_{6}*x_{10}*x_{15}*x_{16}^2, x_{0}*x_{1}*x_{11}^3*x_{12}^2*x_{13}-x_{5}*x_{10}*x_{15}*x_{16}^2,
x_{0}^2*x_{11}^3*x_{12}^2*x_{13}-x_{4}*x_{10}*x_{15}*x_{16}^2, x_{2}*x_{10}*x_{13}*x_{14}^2*x_{15}^3* x_{16}^4-x_{0}^2*x_{11}, x_{1}*x_{10}*x_{12}*x_{13}^2*x_{14}^3*x_{15}^4* x_{16}^5-x_{0}^2,
x_{10}*x_{11}*x_{12}^2*x_{13}^3*x_{14}^4*x_{15}^5* x_{16}^6-x_{0}),$
\\

In future work, we will study how the situation can be improved by  optimising the algorithm for computation of Gröbner bases, for example, the Faugère $F_4$ \cite{FaugereF4} and $F_5$ \cite{FaugereF5} algorithms.

\subsection{Stochastic network design}
\label{subsec:SND}

We have also tested our approach on a small stochastic network design problem, formulated as follows:
\begin{align}
	\label{eq:StochND}
	\min & \sum_{a \in A} c_{a} x_{a}+\frac{1}{N} \sum_{i=1}^{N} \sum_{c \in C} \sum_{a \in A} q_{a c} y_{a c}^{i}  \\
	\text { s.t. }  & \sum_{\substack{a \in A \\ a(0)=v}} y_{a c}^{i}-\sum_{\substack{a \in A \\ a(1)=v}} y_{a c}^{i}=d_{v, c}^{i} \quad  \forall(v, c, i) \in V \times C \times\{1, \ldots, N\}  \\
	& \sum_{c \in C} y_{a c}^{i} \leq u_{a} x_{a} \quad  \forall(a, i) \in A \times\{1, \ldots, N\}\\
	& x_{a} \in\{0,1\}, y_{a c}^{i} \in \nonnegint
\end{align}

Following the same procedure as in Section \ref{subsec:HS2003}, we compare our two algorithms with a commercial solver. Our instances concern a small problem, with three vertices and three arcs. Figure \ref{fig:performance-SND} shows the time taken to compute the opportunity cost matrix for each approach.
\begin{figure}[h]
	\centering
	\includegraphics[scale=0.7]{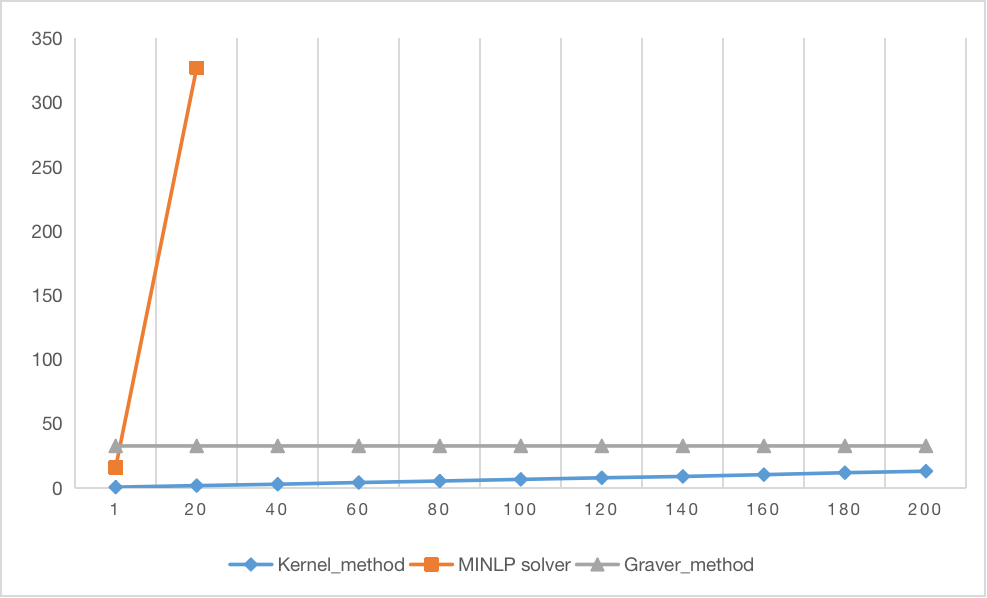}
	\caption{Trends of Computational Costs when the Sub-problem is Fixed  (Stochastic Network Design Problem)}
	\label{fig:performance-SND}
\end{figure}
We once more observe the superior efficiency of the kernel approach when the number of decision variables is small.
%Here we can also see that the kernel algorithm has supreme efficiency over the other two algorithms when one single instance has a comparatively small size.

%%%%%%%%%%%%%%%%%%%%%%%%%%%%%%%%%%%%%%%%
%%%%%%%%%%%%%%%%%%%%%%%%%%%%%%%%%%%%%%%%
%%%%%%%%%%%%%%%%%%%%%%%%%%%%%%%%%%%%%%%%
%%%%%%%%%%%%%%%%%%%%%%%%%%%%%%%%%%%%%%%%

\section{Conclusion}
\label{sec:conclusion}

In this paper, we have presented new algorithms for computing the opportunity cost matrix of \cite{HewittOrtmannRei2022} using algebraic methods using Gröbner and Graver bases. We have also provided mathematical results showing the relationship of Graver bases for deterministic and stochastic integer programs. Our numerical results show that the Gröbner basis is very efficient as the number of scenarios increases if the number of variables of the underlying integer program does not grow too fast. On the other hand, the Graver basis approach is interesting when there are a large number of variables.

In future work, we propose to include more recent methods of computing Gröbner basis, such as the Faugère algorithms. We will also study how to combine the algebraic approach with approximations such as Lagrangian cuts in order to increase further the number of decision variables that our approach can handle.
	
	\vspace{2cm}
	
	\printbibliography

	\ \ 

\end{document}